\documentclass[12pt]{article}
\usepackage{amssymb}
\usepackage{graphicx}
\usepackage{caption}
\usepackage{subcaption}
\usepackage{amsmath,amssymb,amsfonts,amsthm}
\usepackage{color}
\usepackage{amssymb,amsmath,amsfonts,amssymb,amscd}

\newtheorem{theorem}{Theorem}
\newtheorem{corollary}[theorem]{Corollary}
\newtheorem{lemma}[theorem]{Lemma}

\newtheorem{proposition}[theorem]{Proposition}

\newtheorem*{result}{Result}

\newcommand{\eps}{{\epsilon}}

\begin{document}
\title{On conformally invariant CLE explorations}
\author{Wendelin Werner \and Hao Wu}
\date {}

\maketitle

\newcommand{\U}{\mathbb{U}}
\newcommand{\HH}{\mathbb{H}}
\renewcommand{\H}{\mathbb{H}}
\newcommand{\R}{\mathbb{R}}
\newcommand{\N}{\mathbb{N}}

\def\PV{\mathrm{P.V.}}

\begin{abstract}
We study some conformally invariant dynamic ways to construct the Conformal Loop Ensembles with simple loops
introduced in earlier papers by Sheffield, and by Sheffield and Werner. One outcome is a conformally invariant way
to measure a distance of a CLE$_4$ loop to the boundary ``within'' the CLE$_4$, when one identifies all points of each loop.
\end{abstract}

\section {Introduction}

The present paper will be devoted to the study of some
properties of the Conformal Loop Ensembles (CLE) defined and studied by Scott Sheffield in \cite {Sheffield09} and by Sheffield and Werner in \cite {CLE1}.

A simple CLE can be viewed as a random countable collection $(\gamma_j, j \in J)$ of disjoint simple loops in the unit disk that are non-nested (almost surely no loop surrounds another loop in CLE). The paper \cite {CLE1} shows that there are several different ways to characterize them and to construct them. In that paper, a CLE is defined to be such a random family that also possesses two important properties: It is conformally invariant (more precisely, for any fixed conformal map $\Phi$ from the unit disk onto itself, the law of $(\Phi( \gamma_j), j \in J)$ is identical to that of $(\gamma_j, j \in J)$ -- this allows to define the law of the CLE in any simply connected domain via conformal invariance) and it satisfies a certain natural restriction property that one would expect from interfaces in physical models.
It is shown in \cite {CLE1} that there exists exactly a one-parameter family of such CLEs. Each CLE law corresponds exactly to some $\kappa \in (8/3, 4]$ in such a way that for this $\kappa$, the loops in the CLE are loop-variants of the SLE$_\kappa$ processes (these are the Schramm-Loewner Evolutions with parameter $\kappa$ -- recall that an SLE$_\kappa$ for $\kappa \le 4$ is a simple curve with Hausdorff dimension $1+ \kappa /8$), and that conversely, for each $\kappa$ in that range, there exists exactly one corresponding CLE.
Part of the arguments in the paper \cite {CLE1} are based on the analysis of discrete ``exploration algorithms'' of these loop ensembles, where one slices the CLE open from the boundary and their limits (roughly speaking when the step-size of explorations tends to zero).

In the earlier paper \cite {Sheffield09}, Sheffield had pointed out a way to construct a number of random collections of loops, using variants of SLE$_\kappa$ processes. In particular, for any $\kappa \in (8/3, 4]$, he has shown how to construct random collections of SLE$_\kappa$-type loops (or rather ``quasi-loops'' that should turn out to be loops, we will come back to the precise definition later; for the time being the reader can think of these quasi-loops as boundaries of a bounded simply connected component that is a loop with at most one point of discontinuity)
that should be the only possible candidates for the conformally invariant scaling limit of various discrete models, or of level lines of certain continuous models. Roughly speaking,
one chooses some boundary point $x$ on the unit circle (``the root'') and  launches from there a branching exploration tree of SLE processes (or rather target-independent variants of  SLE$_\kappa$ processes called the SLE($\kappa, \kappa-6$) processes) that will trace some loops along the way, that one keeps track of. For each $\kappa$ and $x$, there are in fact several ways to do this. One particular way, that we will refer to as {\em symmetric} in the present paper, is to impose certain ``left-right'' symmetry in the law of the exploration tree, but several
other natural options are described in \cite {Sheffield09}. Hence, for each $\kappa$, the exploration tree is defined via the choice of the root $x$ and the exploration ``strategy'' that describes how ``left-right'' asymmetric the exploration is. These exploration strategies are particularly natural, because they are invariant under all conformal transformations that preserve $x$.
Note also that it is conjectured that the exploration indeed traces a tree with continuous branches, but that this result is not yet proved (to our knowledge). Nevertheless, these
processes indeed trace continuous quasi-loops along the way. So to sum up, once $\kappa$, $x$ and a given strategy are chosen, the Loewner differential equation enables to construct a random family of quasi-loops in the unit disc (and the law of this family a priori depends on $\kappa$, on $x$ and on the chosen strategy).

One can also note that the symmetric strategy is very natural for $\kappa=4$ from the perspective of the Gaussian Free Field, but much less
so from the perspective of interfaces of lattice models. For instance, in the case of $\kappa=3$ viewed as the scaling limit of the Ising model (see \cite {ChSm}), the ``totally asymmetric'' procedure seems more natural.

Based on the conjectured or proved relation to discrete models and to the Gaussian Free Field, Sheffield conjectures in \cite {Sheffield09} that for any given $\kappa \le 4$,
all these random collections of loops traced by the various exploration trees have the same law.

One consequence of the results of Sheffield and Werner in \cite {CLE1} is that this conjecture is indeed true for all symmetric explorations: More precisely, for each $\kappa \in (8/3, 4]$, the law of the random collection of quasi-loops traced by a symmetric SLE$(\kappa, \kappa-6)$ exploration tree rooted at $x$ does not depend on $x$. In fact, their common law is proved to be that of ``the CLE" with SLE$_\kappa$-type loops mentioned in the first paragraph of this introduction, and they can also be viewed (see \cite {CLE1}) as
outer boundaries of clusters of Brownian loop-soups (which proves that the quasi-loops are in fact all loops).
 One main idea in \cite {CLE1} is to study the asymptotic behavior of the
discrete explorations when the steps get smaller and smaller, and to prove that it converges to the above-mentioned symmetric SLE($\kappa, \kappa-6$) process.

The first main result of the present paper can be summarized as follows: {\em For each $\kappa \in (8/3, 4]$, all the random collections of SLE$_\kappa$-type quasi-loops constructed via Sheffield's asymmetric exploration trees in \cite {Sheffield09} have the same law. They all are the CLE$_\kappa$ families of loops constructed in \cite {CLE1}.} The proof of this fact will heavily rely on the results of \cite {CLE1}, but we will try to make our paper as self-contained as possible.

Recall that when one works directly in the SLE-framework, certain questions turn out to be rather natural -- it is for instance possible to derive rather directly the values of certain critical exponents, to compute explicitly probabilities of certain events, or to study questions related to conformal restriction -- while the setting of the Loewner equation does not seem so naturally suited for some other questions. Proving reversibility of the SLE path (that the random curve defined by an SLE from $a$ to $b$ is the same as that defined by an SLE from $b$ to $a$) turns out to be very tricky, see \cite {Zhan,MS2,MS3}. A by-product of \cite {CLE1} is that it provides another proof
of this reversibility in the case where $\kappa \in (8/3, 4]$. In a way, our results are of a similar nature: One obtains results about these asymmetric branching SLE processes without going into fine Loewner chain technology.

The second main point of our paper is to highlight something specific to the case $\kappa=4$ i.e. to CLE$_4$ (recall that this is the CLE that is most directly related to the Gaussian Free Field, see \cite {SchSh, SS, Du, Sheffield11}). In this particular case, it is possible to define a conformally invariant and unrooted (one does not need to even choose a starting point) growing mechanism of loops (the term ``exploration'' that is used in this paper is a little bit misleading, as it is not proved that the growth process is in fact a deterministic function of the CLE, we will discuss this at the end of the paper). Roughly speaking, the growth process that progressively discovers loops is growing ``uniformly'' from the boundary (even if it is a Poisson point process and each loop is discovered at once) and does not require to choose a root. The fact that such a conformally invariant non-local growth mechanism exists at all is quite surprising (and the fact that its time-parametrization as seen from different
points does exactly coincide even more so). It also leads to a conformal invariant way to describe distances between loops in a CLE (where any two loops in a CLE are at a positive distance of each other) and to a new coupling of CLE with the Gaussian Free Field that will be studied in more detail in the subsequent work \cite {SWW},
and to open questions that we will describe at the end of the paper.

\section {Background and first main statement}

In the present section, we recall some ideas, arguments and results from  \cite {Sheffield09, CLE1}, and set up the framework that will enable us to derive our main results in a rather simple way.

\subsection {Bessel processes and principal values}

Suppose throughout this section that $\delta \in (0,1 ]$. It is easy to define the {\em squared Bessel process} $(Z_t ,t \ge 0)$ of dimension $\delta$ started from $Z_0 = z_0 \ge 0$
as the unique solution to the stochastic differential equation
$$ dZ_t = 2 \sqrt  {Z_t} dB_t + \delta dt $$
where $(B_t, t \ge 0)$ is a Brownian motion (note that it is implicit that this solution is non-negative because one takes its square root).

The non-negative process $Y_t  = \sqrt {Z_t}$ is then usually called the Bessel process of dimension $\delta$ started from
$\sqrt {z_0}$. It is not formally the solution to the stochastic differential equation
$$dY_t = dB_t + (\delta-1) \frac {dt}{2Y_t}$$
because it gets an (infinitesimal) upwards push whenever it hits the origin, so that
$Y_t - (\delta-1) \int_0^t ds/(2Y_s)$ is not a martingale (note for instance that when $\delta=1$, the process $Y_t$ is a reflected Brownian motion, which is clearly not the solution to $dY_t = dB_t$). This stochastic differential equation however describes well the evolution of $Y$ while it is away
from the origin, and if one adds the fact that $Y$ is almost surely non-negative,
continuous and that the Lebesgue measure of $\{ t > 0 \ : \ Y_t= 0 \}$ is almost surely equal to $0$, then it does characterize $Y$ uniquely.
Note that the filtration generated by $Y$ and by $B$ do coincide ($B$ can be recovered from $Y$).

Bessel processes have the same scaling property as Brownian motion: When $Y_0=0$, then for any given positive $\rho$, $(Y_t, t\ge 0)$ and $( \rho^{-1} Y_{\rho^2 t}, t\ge 0)$ have the same law
 (this is an immediate consequence of the definition of its squared process $Z$).
Just as in the case of the It\^o measure on Brownian excursions, it is possible to define
 an infinite measure $\lambda$ on (positive) Bessel excursions of dimension $\delta$. An excursion $e$ is a continuous function $(e(t), t \in [0, \tau])$ defined on an interval of non-prescribed length $\tau=\tau(e)$ such that $e(0)=e(\tau)=0$ and $e (t) > 0 $ when $t \in (0, \tau)$. The measure $\lambda$ is then characterized by the fact that for any $x$, the mass of the set of excursions
 $$E_x := \{ e \ : \ \sup_{ s \le \tau }  e(s) \ge x \}$$ is finite, and that if one renormalizes $\lambda$ in such a way that it is a probability measure on $E_x$, then the law of $e$ on $[\tau_x, \tau]$
is that of a Bessel process of dimension $\delta$, started from $x$ and stopped at its first hitting time of the origin.
 The fact that $t \mapsto (Y_t)^{2-\delta}$ is a local martingale when $Y$ is away from the origin (which follows immediately from the definition of the Bessel process $Y$ and from It\^o's formula) shows readily that
 $$ \lambda ( E_x  ) = c x^{\delta -2}$$
 for some constant $c$ that can be chosen to be equal to one (this is the normalization choice of $\lambda$).
Standard excursion theory shows that it is possible to define the process $Y$ by gluing together a Poisson point process $(e_u, u \ge 0)$ of these Bessel excursions of dimension $\delta$, chosen with intensity $\lambda \otimes du$.

Suppose that we are given a parameter $\beta \in [-1, 1]$ and that for each excursion $e$ of the Bessel process $Y$, one tosses an independent coin in order to choose $\eps ( e)= \eps_\beta (e) \in \{-1, 1\}$ in such a way that the probability that $\eps (e) = +1$ (respectively $\eps (e) = -1$) is $(1+\beta)/2$ (resp. $(1- \beta)/2$). Then, one can define a process $X^{(\beta)}$ by gluing together the excursions $(\eps(e) \times e)$ instead of the excursions $(e)$. Note that $|X^{(\beta)}| = X^{(1)}= Y$, but some of the excursions of $X^{(\beta)}$ are negative as soon as $\beta < 1$. The process $X^{(0)}$ is the {\em symmetrized} Bessel process such that $X^{(0)}$ and $-X^{(0)}$ have the same law. Note that (as opposed to $Y=X^{(1)}$) the process $X^{(\beta)}$ is not a deterministic function of the underlying driving Brownian motion $B$ in the stochastic differential equation when $\beta \in (-1,1)$, because additional randomness is needed to choose the signs of the excursions. However,  $B$ is still a
martingale with respect to the
filtration generated by $X^{(\beta)}$ because the signs of the excursions are in a way independent of the excursions.

In the sequel, it will be useful to consider the processes $X^{(\beta)}$ for various values of $\beta$ simultaneously. Clearly, it is easy to first define $Y$ and then to couple all
signs in such a way that for each excursion $e$ of $Y$, $\eps_\beta (e) \ge \eps_{\beta'} (e)$ as soon as $\beta \ge \beta'$; we will implicitly always work with such a coupling.

\medbreak

In the context of SLE processes, it turns out to be essential to try to make sense of a quantity of the type $\int_0^t ds / X_s^{(\beta)}$.
We define, for each excursion, the integral
$$ i(e) := \int_0^\tau ds/ e(s) .$$
It is easy to check that
 $$\lambda (i(e) 1_{E_1}) < \infty,$$ from which it follows using scaling that $i(e) < \infty$ for $\lambda$ almost all excursion $e$.
 Note also that the scaling shows that
$$ \lambda ( \{ e \ : \ i(e) \ge x \} ) = x^{\delta -2} \lambda ( \{ e \ : \ i(e) \ge 1 \} ).$$
It follows that typically, the number of excursions that occur before time $1$ for which $i(e) \in [2^{-n}, 2^{-n+1})$ is of the order of $(2^{-n})^{\delta-2}$, so that their cumulative contribution to $\int_0^t ds/Y_s$ is of the order of $(2^{-n})^{\delta-1}$. If we sum this over $n$, one readily sees that when $t >0$, then
$$ \int_0^t ds / Y_s =  \infty$$
almost surely as soon as $\delta\le 1$ (and this argument can be easily made rigorous)
due to the cumulative contributions of the many short excursions during the interval $[0,t]$. Hence, $ \int_0^t ds/ X_s^{(\beta)}$ can not be defined as a
simple absolutely converging integral.

There are however ways to circumvent this difficulty. The first classical one works for all $\delta \in (0,1]$ but it is specific to the case where $\beta=0$ i.e. to the symmetrized Bessel process $X^{(0)}$. In that case, when one formally evaluates the cumulative contribution to $\int_0^t ds/ X_s^{(0)}$ of the excursions for which $i (e) \in [2^{-n}, 2^{-n+1} )$, then the central limit theorem suggests that one will get a value of the order of $2^{-n} \times (2^{(2- \delta)n})^{1/2} =2^{-\delta n/2}$; when one then sums over $n$, one gets an almost surely converging series.
This heuristic can be easily be made rigorous, and this shows that one can define a process $I^{(0)}_t$ that one can informally interpret as $\int_0^t ds/X_s^{(0)}$ (even though this last integral does not converge absolutely). Another possible way to characterize this process is that it is the only process such that:
\begin {itemize}
\item $t \mapsto I_t^{(0)}$ is almost surely continuous and satisfies Brownian scaling.
\item  $dI_t^{(0)} - dt/X_t^{(0)}$ is zero on any time-interval where $X^{(0)}$ is non-zero.
\item The process $I^{(0)}$ is a deterministic function of the process $X^{(0)}$.
\end {itemize}
Let us reformulate and detail our first approach to $I_t^{(0)}$ in a way that will be useful for our purposes. Suppose that $r >0$ is given and small. We denote by $J_r$ the set of times that belong to an excursion of $Y$ away from the origin, that has time-length at least $r^2$ (we choose $r^2$ in order to have the same scaling properties as for the height and $i(e)$). Then, because the integral $\int ds/e(s)$ on each individual excursion is finite, we see that it is possible to define without any difficulty the absolutely converging integral
$$ I_t^{(0,r)} :=  \int_0^t ds 1_{s \in J_r} / X_s^{(0)}.$$
Then,  as $r \to 0$ the continuous process $I^{(0,r)}$ converges to the continuous process $I^{(0)}$. More rigorously

\begin {lemma}
When $n \to \infty$, then on any compact time-interval, the sequence of continuous functions $I^{(0, 1/2^{n})}$ converges almost surely to a limiting continuous function $I^{(0)}$.
\end {lemma}

\begin {proof}
Let $\tau = \tau (r_0)$ denote the end-time of first excursion that has time-length at least $r_0^2$ (here $r_0$ should be thought of as very large, so that this time, which is greater than $r_0^2$, is large too). It suffices to prove the almost sure convergence on the interval $[0, \tau]$ (as any given compact interval is inside some interval $[0, \tau]$ for large enough $r_0$).

Let us suppose that $n \ge m$. Notice that the process
$t \mapsto I_t^{(0, 1/2^n)} - I_t^{(0, 1/2^m)}$ is monotonous (i.e. non-increasing or non-decreasing) on each excursion of $Y$, so that
$$\sup_{ t \le \tau } ( I_t^{(0, 1/2^n)} - I_t^{(0, 1/2^m)} )^2  = \sup_{t \le \tau, Y_t= 0}  ( I_t^{(0, 1/2^n)} - I_t^{(0, 1/2^m)} )^2.$$

Next we define the $\sigma$-field ${\mathcal F}_0$ generated by the knowledge of all excursions $| e |$, but not their signs. If we condition on ${\mathcal F}_0$ and look at the value of
 $I_t^{(0, 1/2^n)} - I_t^{(0, 1/2^m)}$ at the end-times of the excursions of length greater than $2^{-n}$, we get a discrete martingale.
 From Doob's $L^2$ inequality, we therefore see that almost surely,
$$
E \left( \sup_{ t \le \tau } ( I_t^{(0, 1/2^n)} - I_t^{(0, 1/2^m)} )^2    \mid {\mathcal F}_0 \right)  \le 4 E\left( (I_\tau^{(0, 1/2^n)} - I_\tau^{(0, 1/2^m)} )^2  \mid {\mathcal F}_0 \right).
$$

The right-hand side is in fact the mean of the square of a series of symmetric random variables of the type $\sum \eps_j i_j$ for some given $i_j$ and coin-tosses $\eps_j$. Therefore, it is equal to
$4 \sum_e i(e)^2 $
where the sum is over all excursions appearing before time $\tau$, corresponding to times in $J_{1/2^{n}} \setminus J_{1/2^{m}}$.  By simple scaling, the expectation of this quantity is equal to a constant times $2^{-m \delta} - 2^{-n\delta}$, so that finally
$$
E \left( \sup_{ t \le \tau } ( I_t^{(0, 1/2^n)} - I_t^{(0, 1/2^m)} )^2 \right) \le C ( 2^{-m \delta} - 2^{-n\delta}) \le C 2^{-m \delta}.$$
It then follows easily (via Borel-Cantelli) that almost surely, the function $t \mapsto I_t^{(0, 1/2^n)}$ converges uniformly as $n \to \infty$ on the time-interval $[0,\tau]$ (and that the limiting process $I^{(0)}$ is continuous).
\end {proof}

\medbreak
This construction of $I^{(0)}$ can not be directly extended to the case where $\beta \not=0$. Indeed,  the cumulative contributions of those excursions of $X^{(\beta)}$ for which $i(e) \in [2^{-n}, 2^{1-n})$  is then of the same order of magnitude than when $\beta=1$ (the previously described case where one looks at the integral of $1/Y_s$). A solution when the dimension of the Bessel process is smaller than $1$, is to compensate the explosion of this integral appropriately. Let us first describe this in the case where $\beta=1$ (i.e. $X^{(\beta)} = Y$ is the non-negative Bessel process). As for instance explained in \cite {Sheffield09}, Section 3,
 it is possible to characterize the principal value $I_t=I_t^{(1)}$ of the integral of $1/Y_t$ as the unique process such that:
\begin {itemize}
\item $t \mapsto I_t$ is almost surely continuous.
\item  $dI_t - dt/Y_t$ is zero on any time-interval where $Y$ is non-zero.
\item $(I_t, Y_t)$ is adapted to the filtration of $Y$ and satisfies Brownian scaling.
\end {itemize}

Let us describe how to construct explicitly this process $I_t$.
For any very small  $r$, recall the definition of the time-set $J_r$, and define  $N_r (t)$ as the number of excursions of time-length at least $r^2$ that $Y$ has completed before time $t$. Simple scaling considerations show that (for fixed $t$), $N_r (t)$ will explode like (some random number times) $r^{\delta-2}$ as $r \to 0$.

Just as before, there is no problem to define the absolutely converging integral
$$  \int_0^t \frac {1_{s \in J_r} ds}{Y_s} .$$
But, as we have already indicated, when $\beta \not= 0$ this quantity tends to $\infty$ when $r$ tends to $0$.  One option is therefore to consider the quantity
\begin {equation*}
K_t^r :=  \int_0^t \frac {1_{s \in J_r} ds}{Y_s}  - C r  N_r (t)
\end {equation*}
where
$$C r := \frac {\lambda ( i(e) 1_{\tau (e) \ge r^2}) }{ \lambda ( 1_{\tau (e) \ge r^2 })}$$  is the mean value of the integral of $1/e$
for an excursion conditioned to have length greater than $r^2$. Note that
$$ C=\lambda(i(e)1_{\tau(e)\ge 1})/\lambda(1_{\tau(e)\ge 1})$$
 is a constant that does not depend on $r.$
When $r \to 0$,  $r N_r (t)$ explodes like $r^{\delta -1}$, but nevertheless:
\begin {lemma}
 As $n \to \infty$, the process $K^{1/2^{n}}= (K_t^{1/2^{n}}, t \ge 0)$ does almost surely converge uniformly on any compact time-interval
to some continuous limiting  process $I^{(1)}$.
\end {lemma}

\begin {proof}
The proof goes along similar lines as in the case $\beta=0$, but there are some differences. Let us  prove again almost sure convergence on each $[0,\tau]$ for each given $r_0$.

First, let us notice that for each $r>r'$, the quantity $ K_t^r - K_t^{r'}$ is constant, except on the excursion-intervals of time-length between $r'$ and $r$.
If we follow the value of this quantity only at (the discrete set of) end-times of excursions of length greater than $r'$,  we get a discrete martingale.
Doob's inequality and scaling, just as before, imply that
$$ E \left( \sup_{t \le \tau :  Y_t = 0 } ( K^r_t - K_t^{r'})^2  \right) \le  E \left(  ( K^r_\tau - K_\tau^{r'})^2  \right) .$$
The right-hand side can be viewed as the sum of a geometric number of zero-mean random variables with bounded second moment, and it
follows from scaling that it can be bounded by
$$c r_0^{2-\delta} r^\delta.$$
It follows that almost surely,  the function $K^{1/2^{n}}$ converges uniformly as $n \to \infty$ on the set $\{ t \le \tau \ : \ Y_t = 0 \}$. The definition of $K^{r}$ then yields that this almost sure uniform convergence takes place on all of $[0,\tau]$ (just because the supremum of the integral of $1/Y_s$ over all excursions of length greater than $r$ before $\tau$ goes to $0$ as $r \to 0$).

On the other hand, if we slightly modify $K^r$ on each excursion interval by adding a linear function that makes it continuous on the closed support of the excursion in order to compensate the $-Cr$ jump of $K^r$ and the end-time of the excursion, one obtains a continuous function $\tilde K^r$ such that
$| K^r_t - \tilde K^r_t | \le Cr$ for each $t$. It follows that  almost surely
$ \tilde K^{1/2^{n}}$ converges uniformly on any compact time-interval to the same limit as $K^{1/2^{n}}$. As the functions $\tilde K^r$ are continuous, it follows that this limit is almost surely a continuous function of time.
\end {proof}

\medbreak

For any $\beta$ (and as long as $\delta \in (0,1)$), the very same idea can be used to define a process $I^{(\beta)}$ associated to $X^{(\beta)}$ instead of $Y$, as the limit when $r \to 0$ of the process
$$  \int_0^t \frac {1_{s \in J_r} ds}{X_s^{(\beta)}}  - C r \beta   N_r (t).
$$

\medbreak

Let us describe in more detail a variant of the previous construction that will be useful for our purposes. Suppose that one is working with the
coupling of all processes $X^{(\beta)}$ (for fixed $\delta \in (0,1)$).  We then define the process
$$X^{(\beta, r)}_t := X^{(\beta)}_t 1_{t \in J_r} + X^{(0)}_t 1_{t \notin J_r}.$$
In other words, we replace $X^{(\beta)}$ by $X^{(0)}$ on all excursions of length smaller than $r^2$.
 Clearly, this makes it possible to make sense of the continuous process
$$ \int_0^t \left( \frac {1}{X_s^{(\beta, r)} } - \frac 1 {X_s^{(0)}} \right) ds$$
(because only the times in $J_r$ i.e. in the macroscopic excursions will contribute).
We can therefore define the process
$$ I_t^{(\beta, r)}:=  I_t^{(0)}  +    \int_0^t  1_{s \in J_r} \left( \frac {1}{X_s^{(\beta, r)} } - \frac 1 {X_s^{(0)}} \right) ds - \beta C r N_r (t).$$
This process $I^{(\beta, r)}$ follows exactly the evolution of $I^{(0)}$ except that some excursions of length greater than $r^2$ are sign-changed (and on these excursions $I_t^{(\beta, r)} + I_t^{(0)}$ is constant), and that at the end of each of those excursions, it makes a small jump of $-\beta C r$.

Then, almost surely, when $r \to 0$, the process $I^{(\beta, r)}$ converges uniformly on any compact time-interval to a process $I^{(\beta)}$ because the two processes
$$ I^{(0)}_t - \int_0^t 1_{s \in J_r} ds / X_s^{(0)}$$
and
$$ \int_0^t 1_{s \in J_r} ds / X_s^{(\beta)} - \beta C r N_r (t) - I^{(\beta)}_t $$
do almost surely uniformly converge to zero on any given compact interval.

\medbreak

The previous definition of $I^{(\beta)}$ can not be directly adapted to the case $\delta=1$.
However, one notes that for any real $\mu$, the process
$I_t^{<\mu>} := I_t^{(0)} +  \mu \ell_t$, where $\ell$ is the local time at $0$ of $X$ does also satisfy the Brownian scaling property and that $dI_t^{<\mu>} = dt / X_t$  on all intervals where $X$ is non-zero. This process $I^{<\mu>}$ can in fact again be approximated via $N_r(t)$ (using the classical approximation of Brownian local time); more precisely, it is the limit as $r \to 0$ of the process
$$ I_t^{<\mu, r >} := I_t^{(0)} + \mu r N_r(t) .$$

\subsection {From Bessel processes to ensembles of (quasi-)loops}

We now recall (mostly from \cite {Sheffield09}) how to use the previous considerations in order to give an SLE based construction of the loop ensembles.
If we work in the upper half-plane $\H$, then Loewner's construction shows that as soon as one has defined a continuous real-valued function
$(w_t, t \ge 0)$, one can define a two-dimensional ``Loewner chain'' (that in many cases turns out to correspond to a two-dimensional path) as follows:
For any $z \in \overline \H$, define the solution $(Z_t=Z_t(z))$ to the ordinary differential equation
$$ Z_t = z + \int_0^t \frac {2 ds}{Z_s-w_s}.$$
This equation is well-defined up to a (possibly infinite) explosion/swallowing time $T(z) = \sup \{ t \ge 0 \ : \ \inf \{ |Z_s - w_s | \ : \ s \in [0,t ) \} > 0 \}$.
For each given $t$, the map $g_t: z \mapsto Z_t(z)$ is a conformal map from some subset $H_t$ of $\H$ onto $\H$ such that $g_t (z) - z = o(1)$ when $z \to \infty$.
One defines $K_t = \overline {\HH \setminus H_t}$; the Loewner chain usually means the chain $(K_t, t \ge 0)$.

When $w_t$ is chosen to be equal to $\sqrt {\kappa} B_t$, where $B$ is a standard real-valued Brownian motion, then this defines the SLE$_\kappa$ processes, that turn out to be simple curves as soon as $\kappa \le 4$.
The so-called SLE($\kappa, \kappa-6$) processes are variants of SLE$_\kappa$ with a particular target independence property first pointed out in \cite {WilsonSchramm}. More precisely, suppose that one considers the joint evolution of two points $(W_t, O_t)$ in $\R$ started  from  $(W_0, O_0)$ with $W_0 \not= O_0$, and described by
\begin {equation}
 dO_t = \frac {2dt}{O_t - W_t} \hbox { and } dW_t = \sqrt {\kappa} dB_t + \frac {\kappa-6}{W_t - O_t} dt
 \label {evol}
 \end {equation}
as long as $W_t \not= O_t$ (where $B$ is a standard Brownian motion). Then, one can use the random function $W$ as the driving function of our Loewner chain, which is this SLE$(\kappa, \kappa-6)$. There is no difficulty in defining the process $(W,O)$ as long as $W_t$ does not hit $O_t$, but more is needed to understand what happens after such a meeting time.

Note that if one writes $X_t = (W_t - O_t)/ \sqrt {\kappa}$, then
$$ dX_t = dB_t + \frac {\kappa-4}{\kappa X_t} dt $$
so that $X$ evolves like a Bessel process of dimension $$\delta= 3- \frac {8}{\kappa} \in (0, 1]$$ when $\kappa \in (8/3, 4]$, and that $dO_t$  is a constant multiple of $dt/X_t$ as long as $X_t \not=0$.
Furthermore, the knowledge of $X_t$ and of $O_t$ enables to recover $W_t = O_t + \sqrt {\kappa} X_t$.

This gives the following options to define a driving process $(W_t, t\ge 0)$ at all non-negative times (even for $W_0=O_0=0$):
\begin {itemize}
\item When $\delta\in (0,1)$ (i.e., $\kappa \in (8/3, 4)$) and $\beta \in [-1, 1]$:
Define first one of the Bessel processes $X^{(\beta)}$ as before started at $0$, and its corresponding process $I^{(\beta)}$.
Then define $O_t^{(\beta)}= 2 \sqrt{\kappa} I_t^{(\beta)}$ and
$$W_t^{(\beta)} =  \sqrt{\kappa} X_t^{(\beta)} + O_t^{(\beta)} = \sqrt{\kappa} X_t^{(\beta)} + 2 \sqrt{\kappa} I_t^{(\beta)}.$$
\item When $\delta=1$ (i.e., $\kappa = 4$) and $\mu \in \R$, then, define $O_t^{<\mu>} = 4 I_t^{<\mu>}$ and
$$W_t^{<\mu>} = 2 B_t + O_t^{<\mu>} = 2 B_t +  4 I_t^{<\mu>}$$  where $B$ is standard one-dimensional Brownian motion.
\end {itemize}

In all these cases, one constructs a couple $(W_t, O_t)$ that satisfies the Brownian scaling property and that evolves according to (\ref {evol}) when $W_t \not= O_t$. The process $(W_t, t \ge 0)$
defines a Loewner chain $(K_t, t \ge 0)$ from the origin to infinity in the upper half-plane. More precisely, for each $t$, $H_t := \HH \setminus K_t$ is the preimage of $\HH$ under the conformal map $g_t$ characterized by the fact that for all $s \le t$,
$$ g_0 (z) = z \hbox { and } \partial_s g_s (z) = \frac {2}{g_s (z) - W_s }$$
(see e.g. \cite {Lawler} for background).
The Brownian scaling property shows that this Loewner chain is invariant (in law) under scaling (modulo time-parametrization).
 This makes it possible to also define (via conformal invariance) the law of the Loewner chain in $\H$ from $0$ to some $u \in \R$ (and more generally from any boundary point to any other boundary point of a simply connected domain) by considering the conformal image of the previously defined chain from $0$ to infinity under a conformal map from $\H$ onto itself that maps $0$ onto itself, and $\infty$ onto $u$.
 
 In the sequel, when we will refer to ``a SLE($\kappa, \kappa-6$) process'', we will implicitely mean such a chordal chain, for some $\kappa \in (8/3,4]$, and some choice of $\beta$ (if $\kappa \in (8/3,4)$) or $\mu$ (when $\kappa = 4$). 

\medbreak
 All these SLE($\kappa, \kappa-6$) processes are of particular interest because of their target-independence property: Up to the first time at which the Loewner chain disconnects $u$ from infinity, the two Loewner chains (from $0$ to $\infty$, and from $0$ to $u$) have the same law (modulo time-change).
When $O_t - W_t$ is not equal to $0$, the fact that the local evolution of chain is independent of the target point is derived (via It\^o formula computations for $(O_t, W_t)$) in \cite {WilsonSchramm}. Note that the two evolutions match up to a time-change only, because time corresponds to the size of the Loewner chain seen from either infinity or from $u$. In order to check that target-independence remains valid at all times, one needs to check that the ``local push'' rule that is used in order to define $I_t^{(\beta)}$ (or $I_t^{<\mu>}$ when $\kappa=4$) and then $O_t$ is also the same (modulo the time-change) for both processes. This is basically explained in Section 7 of \cite {Sheffield09}.

\medbreak

In fact, if we formally replace the Bessel process $X$ by just one Bessel excursion $e$, the procedure defines an SLE ``pinned loop'' or SLE ``bubble'' i.e. a continuous simple curve in $\HH \cup \{ 0\}$
that passes through the origin. More precisely, let us start with the excursion $(e(s), s \le \tau)$  and use the driving function
$$ w_t  := \sqrt {\kappa} e(t) + 2 \sqrt {\kappa} \int_0^t ds / e(s) $$
to generate the Loewner chain. It is shown in \cite {CLE1} that for almost all (positive) Bessel excursion $e$ (according to the Bessel excursion measure that we have denoted by $\lambda$), this defines a such a simple loop $\gamma(e)$ in the upper half-plane, that starts and ends at the origin (and $H_t = \HH\setminus \gamma (0,t]$). The time-length $\tau (e)$ of the loop corresponds to the half-plane capacity seen from infinity of $\gamma$. The infinite measure on loops $\gamma$ that one obtains when starting from the infinite measure $\lambda$ on Bessel excursions is referred to as the one-point pinned measure in \cite {CLE1}.

Hence, for each excursion $e$ of $X$, corresponding to the excursion interval $[t_- (e), t_+(e)]$ (with $t_+ - t_- = \tau(e)$), one can define the preimage of $\gamma (e)$ under the conformal map $g_{t_-}$
of the Loewner chain. It is not clear at this point that this is a proper continuous loop in $\H$ because we do not know whether $g_{t_-}^{-1}$ extends continuously to the origin, but we already know that it is almost a loop: In particular, the preimage of $\gamma (e) \setminus \{ 0 \}$ is a simple curve such that its closure disconnects some interior domain from an outer domain in $H_t$.
 In the sequel, we will refer to this as a quasi-loop (mind that in \cite {Sheffield09}, this is called a ``conformal loop'' and that the term quasi-simple loops is used for something different).
 One of the consequences of \cite {CLE1} is that {\em in the symmetric case}, all these quasi-loops are in fact loops
 (here one uses the alternative construction using Brownian loop-soup clusters).
 For the other cases, we shall see that it is also the case, but at this stage of the proof, we do not know it yet.

\medbreak

As explained in \cite {Sheffield09}, the target-independence makes it possible to define (for each version of the SLE($\kappa, \kappa-6$) that we have defined, and that we will implicitely keep fixed in the coming three paragraphs)  a ``branching SLE'' structure starting from $0$ and aiming at a dense set of points in the upper half-plane. Let us first describe the process targeting 
$i$, until it discovers a quasi-loop around $i$ (this is the ``radial'' SLE($\kappa, \kappa-6$); we choose here not to introduce the radial Loewner equation but to explain 
this radial process via the chordal setting): 
Consider a chordal SLE($\kappa, \kappa-6$) from $0$ to $\infty$ until time $t_1$, which is the first end-time of an excursion of $X$ after the first moment at which the Loewner chain reaches the (half)-circle of radius $1/2$ away from the origin. One has two possibilities, either at $t_1$, the Loewner chain has traced a quasi-loop around $i$ (and it can be only the quasi-loop that the chain had started to trace when reaching the circle, and $t_1$ is the end-time of the corresponding excursion)  or not. In the latter case, it means that $i$ is still in the remaining to be explored unbounded simply connected component $H_{t_1}$ of the chain at time $t_1$. At this time, the chain is growing at a prime end (i.e. loosely speaking, a boundary point) of $H_{t_1}$. We can now consider the conformal transformation $F_1$ of $H_{t_1}$ onto $\H$ that maps this prime end onto the origin, and keeps $i$ fixed (another way to describe $F_{1}$ is to say that it is the composition of $g_{t_1}$ with the Moebius transformation of the upper half-plane that maps $g_{t_1}(i)$ onto $i$ and $W_{t_1}$ onto $0$). Then, we repeat the same procedure again: Grow an SLE($\kappa, \kappa-6$) from the origin in $\H$ until the first end-time of an quasi-loop that touches the circle of radius $1/2$ etc., and we look at its preimage under $F_1$ (this can be interpreting as a switch of chordal target at $t_1$, this pre-image chain now aims at $F_1^{-1} ( \infty)$ instead of $\infty$). After a geometric number of iterations, one finds a chordal SLE($\kappa, \kappa-6$) in $\H$ that catches $i$ via a quasi-loop that intersects the circle of radius $1/2$. The concatenation of the preimages of these Loewner chains is what is called the radial SLE($\kappa, \kappa-6$) targeting $i$, and the preimage of this quasi-loop that surrounds $i$ is the quasi-loop $\gamma (i)$ defined by this radial SLE. 

In a similar way, one can define for each given $z \in \H$, the radial SLE($\kappa, \kappa-6$) targeting $z$, and the quasi-loop $\gamma (z)$ that it discovers.
The target-independence property of chordal SLE($\kappa, \kappa-6$) can be then used in order to see that for any $z_1,\ldots, z_n$, it is possible to couple all these radial explorations in such a way that for any $k \not= j$, the 
explorations targeting $z_k$ and $z_j$ coincide as long as $z_k$ and $z_j$
remain in the same connected component, and that they are conditionally independent after the moment at which $z_k$ and $z_j$ are disconnected from each
other. Hence, one can define a process $(\gamma (z), z \in \HH)$ of quasi-loops via the law of its finite-dimensional marginals, that has the property that for all $z, z'$, on the event where $\gamma (z)$ surrounds $z'$, one has $\gamma (z')= \gamma (z)$. The countable collection of loops that is defined for instance via the loops that surround points with rational coordinates is the CLE constructed by 
Sheffield \cite {Sheffield09} associated to this value of $\kappa$ and $\beta$ (or $\mu$).

Let us repeat that the law of this family of quasi-loops is characterized by its finite-dimensional distributions (i.e. the laws of the families $(\gamma (z_1), \ldots, \gamma (z_k))$ for any finite set $\{ z_1, \ldots , z_k\}$ in the upper half-plane). A convenient topology to use in order to define these random quasi-loops is, for the loop $\gamma(z)$, to use the Carath\'eodory topology for the inside of the quasi-loop as seen from $z$. In the sequel, when we will say that a sequence of CLE's converges in law to another CLE in the sense of finite-dimensional distributions, we will implicitly be using this topology. 

\medbreak

Note that a quasi-loop will be traced clockwise or anti-clockwise, depending on the sign of the corresponding excursion. For instance, for $\beta=1$, all quasi-loops are traced anti-clockwise.

Let us summarize what these procedures define:
\begin {itemize}
\item
When $\kappa \in (8/3, 4)$, for each $\beta \in [-1, 1]$ and for each boundary point $x$ (corresponding to our choice of starting point in the previous setting) a random family of quasi-loops that we can denote by CLE$_\kappa^{\beta} (x)$.
\item
When $\kappa= 4$, for each $\mu \in \R$ and each boundary point $x$, a random family of quasi-loops that we denote by CLE$_{4, \mu} (x)$.
\end {itemize}

The construction and the target-independence ensures that these CLE's are conformally-invariant (this follows basically from the target-independence property, see \cite {Sheffield09}),
i.e. for each given $x$, $\kappa$ and $\beta$ (or $\mu$), and any Moebius transformation of the upper half-plane, the image of a
 CLE$_\kappa^{\beta} (x)$ under $\Phi$ is distributed like a CLE$_\kappa^{\beta} (\Phi(x))$.

 \medbreak
 
In \cite {CLE1}, it is proved that the law of a CLE$_\kappa^0 (x)$ does not depend on $x$, and that the law of CLE$_{4,0}(x)$ does not depend on $x$
(note that these are the families of quasi-loops constructed via the {\em symmetric} exploration procedure). 
Furthermore, for these ensembles, all quasi-loops are almost sure plain loops (the proof of these last facts are  based on the fact that these loop ensembles are the only family of loops that satisfy some axiomatic properties, and that they can also be constructed as boundaries of clusters of Brownian loops).

\section {The asymmetric explorations}

The present section is devoted to the proof of the following proposition and to its analogue for $\kappa=4$, Proposition \ref {eqvincle2}):

\begin{proposition}\label{eqvincle}
For all given $\kappa \in (8/3, 4)$, the law of CLE$_\kappa^\beta (x)$ does depend neither on $x$ nor on $\beta$.
\end{proposition}

Note that this also proves that in these ensembles, all quasi-loops are indeed loops because we already know that CLE$_{\kappa}^0 (x)$ almost surely consists of loops. Recall also that we know that this CLE$_\kappa^0 (x)$ does not depend on $x$ (in the present section, we will simply refer to it as the CLE$_\kappa$). It is therefore enough to check that in the upper half-plane, the laws of  CLE$_\kappa^\beta (0)$ and of CLE$_\kappa^0 (0)$ coincide, ie., that for any given $z_1, \ldots, z_n$ in the upper half-plane,
the joint law of $(\gamma (z_1), \ldots, \gamma (z_n))$ is the same for these two CLEs.

\medbreak

Let us first make the following observations:

\begin {enumerate}

\item
A by-product of the characterization/uniqueness of CLE derived in \cite {CLE1} is the fact that the image measure $\mu^0$ on pinned loops
$\gamma (e)$ (i.e. the image measure of $\lambda$ under $e \mapsto \gamma (e)$) is invariant under the symmetry with respect to the imaginary axis ($x+iy \mapsto -x + iy$). In other words, it is the same as the image measure on loops defined by $e \mapsto \gamma (-e)$
if one forgets about the time-parametrization of the loops (this is closely related to the reversibility of SLE paths first derived by Zhan \cite {Zhan} by other more direct means).
Note that the time-length $\tau (e)$ is both the half-plane capacity of $\gamma (e)$ and of $\gamma (-e)$ for a given $e$.

\item
Suppose now that we consider a {\em symmetric} SLE($\kappa, \kappa-6$) with driving function $W$, that we stop at the first time $T_1$ at which it completes an excursion of length greater than $r^2$. This defines a certain family of loops in the unit disc via the procedure described above. If we are given some sign $\epsilon_1 \in \{ -1, +1 \}$ (independent of the process $W$), we can decide to modify the previous process $W$ into another process $\hat W$, by just (maybe) changing the sign of the final excursion before $T_1$ into $\epsilon_1$ (it may be that we do not have to change it in order to have it equal to $\epsilon_1$). Then, clearly, the law of $(\hat W, t \le T_1)$ is absolutely continuous with respect to that of $(W_t, t \le T_1)$, and it therefore also defines almost surely a family of loops in the unit disk. The previous item shows in fact that the collection of
loops defined by these two processes up to $T_1$ have exactly the same law (because changing the sign of one final excursion does not change the distribution of the corresponding loop).
Furthermore, the law of the collection of connected components of the complement of this symmetric SLE($\kappa, \kappa-6$) at that moment are clearly identical too.

\item
Suppose that $T$ is some stopping time for the driving function $(W_t, t \ge 0)$ of the symmetric SLE($\kappa, \kappa-6$) process started from the origin in the half-plane.
Suppose furthermore that this stopping time is chosen in such a way that almost surely $W_T = O_T$. Then, we know that the process $(W_{T+t} - W_T, t \ge 0)$ is distributed exactly as $W$ itself, and furthermore, it is independent of $(W_t, t \le T)$.
This means that the conditional law given $(W_t, t \le T)$ of the non-yet explored loops is just a CLE$_\kappa$ in the yet-to-be explored domain. This makes it possible to change the starting point of the upcoming evolution, because we know that the law of the loops defined by the branching symmetric SLE($\kappa, \kappa-6$) is independent of the chosen starting point. In particular, if we consider an increasing sequence of stopping times $T_n$ (and $T_0=0$) such that
$T_n \to \infty$ almost surely and $W_{T_n}= O_{T_n}$ for each $n$, and define the process
\begin {equation}
\label{cstjumps}
\hat W_t = W_t + c N_t
\end {equation}
where $c$ is some constant and $N_t = \max \{ n \ge 0 \ : \ T_n \le  t \}$, the planar  loops associated with the excursions intervals of the Bessel process $X$ will be distributed according to loops in a CLE$_\kappa$. More precisely, for each $n$, if one samples (conditionally on the process up to $T_n$) independent CLE$_\kappa$'s in the remaining unexplored connected components created by the Loewner chain at time $T_n$, and considers the union of the obtained loops with the loops that have been discovered before $T_n$, then one gets a full CLE$_\kappa$ sample.
\end {enumerate}

Let us now combine the previous facts in the case where $\kappa <4$ (i.e., $\delta <1$).
Suppose now that $r >0$ is fixed. Consider a symmetrized Bessel process $X^{(0)}$ and the corresponding driving function $W^{(0)}$. Define the stopping times
$T_n (r)$ as the end-time of the $n$-th excursion of length at least $r^2$ of $|X^ {(0)}|$.  Up to time $T_1 (r)$, we perform the exploration using the driving function $W_t^{(0)}$ (that is defined using the symmetrized Bessel process) except that the sign of the last excursion may have been changed depending on the sign $\eps_1 (r)$. Note that we have just recalled that this procedure defines exactly loops of a CLE$_\kappa$.
Note  that $T_1 (r)$ is the end-time of an excursion and corresponds exactly to the completion of a CLE loop. Then, we force a jump of
 $-\beta C r 2 \sqrt {\kappa} $ of the driving function, and continue from there until time $T_2 (r)$ by following the dynamics of $W_t^{(0)}$ (possibly changing the sign of the last excursion before $T_2 (r))$. At $T_2 (r)$, we again wake a jump of $-\beta C r 2 \sqrt {\kappa}$.
Combining the previous two items shows that for any $n$ and $r$, the following procedure constructs a CLE sample:
\begin {itemize}
 \item Sample the previous process until time $T_n (r)$. Keep all the loops corresponding to the excursions of $X$.
 \item In each connected components created along the way by this Loewner chain (one can view these connected components as those of the complement of the closure of the union
 of the interior of all the loops created by the chain -- because these loops are ``dense'' in the created chain) as well as in the remaining unbounded component, sample independent CLE$_\kappa$'s.
\end {itemize}
Note that the result still holds, if for some given $r$, one stops the process at a stopping time which is the end-time of some loop (not necessarily a $T_n(r)$).

\medbreak

Let us now look at the driving function of the previously defined Loewner chain. It is exactly the one that would obtain if the signs of the excursions of length larger than $r^2$ are those given by $\eps_1 (r), \eps_2 (r)$ etc., and one also puts in the negative jumps at each time $T_n(r)$. This corresponds exactly to the driving function
$$ W^{(\beta, r)}_t := \sqrt {\kappa} X_t^{(\beta, r)} + 2 \sqrt {\kappa} I_t^{(\beta, r)}.$$

\medbreak
\begin{figure}[ht!]
\begin{center}
\begin{subfigure}[b]{0.45\textwidth}
\centering
\includegraphics[width=\textwidth]{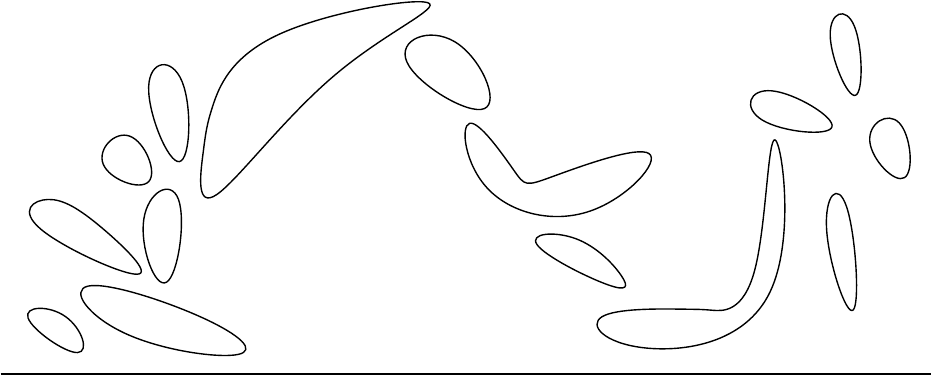}
\caption{Loops generated by $W^{(\beta)}$.}
\end{subfigure}
\bigbreak
\begin{subfigure}[b]{0.45\textwidth}
\centering
\includegraphics[width=\textwidth]{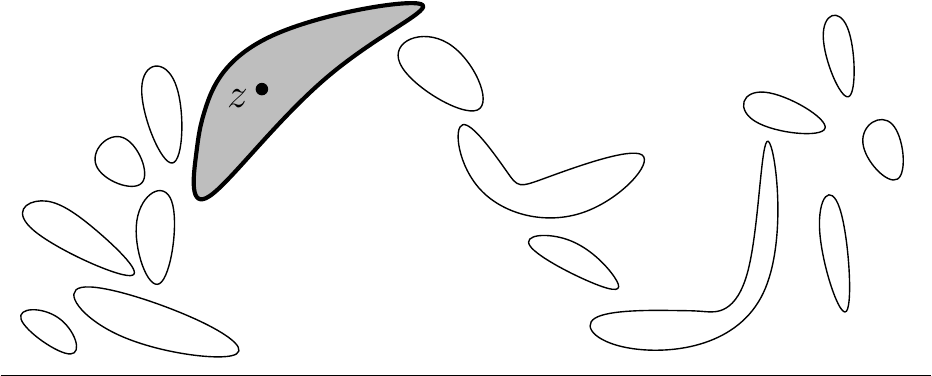}
\caption{$z$ is contained in a loop.}
\end{subfigure}$\quad\quad$
\begin{subfigure}[b]{0.45\textwidth}
\centering
\includegraphics[width=\textwidth]{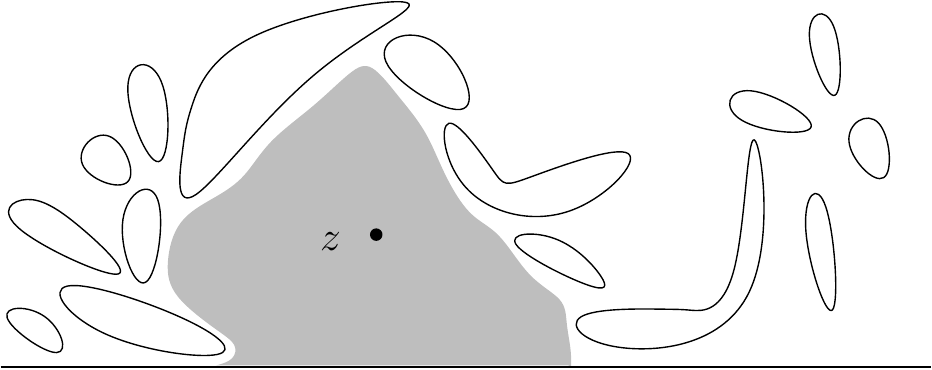}
\caption{The connected component containing $z$.}
\end{subfigure}
\end{center}
\caption{\label{fig::loop_cc_containing_z} Two scenarios for $z$.}
\end{figure}

We now need to control what happens when $r \to 0$. A first important observation is that, as argued in the previous sections, the function
$W^{(\beta, r)}$ converges uniformly towards $W^{(\beta)}$ on any compact interval. On the other hand, we have just described a relation between the Loewner chain/the loops generated by $W^{(\beta, r)}$ and CLE$_\kappa$. Our goal is now to deduce a similar statement for the Loewner chain generated by $W^{(\beta)}$. For that chain, let $H_t^{(\beta)}$ denote the complement of the chain at time $t$ (using the usual notations, and the time-parametrization of $W^{(\beta)}$). 

For simplicity, let us first focus only on the law of the loop that surrounds one fixed interior point $z\in\HH,$ in a CLE$_{\kappa}^{\beta} (0)$. Two scenarios can occur. Either for the Loewner chain (without branching, and targeting infinity) generated by $W^{(\beta)}$, $z$ is at some point swallowed by a quasi-loop (that we then call $\gamma^{(\beta)} (z)$) or not. Let $A_1^{(\beta)}(z)$ and $A_2^{(\beta)}(z)$ denote these two events. If $A_2^{(\beta)}(z)$ holds, then let 
$\tau^{(\beta)}(z) = \sup \{ t > 0 \ : \ z \in H_t^{(\beta)} \}$. 
 
When $A_1^{(\beta)}(z)$ holds, the quasi-loop $\gamma^{(\beta)}(z)$ corresponds to a time-interval $(t_-(e), t_+(e))$ for some excursion $e$, that in turn corresponds (for each $r$) to
some loop $\gamma^{(\beta, r, *)}$ traced by the Loewner chain driven by $W^{(\beta, r)}$. The almost sure convergence of $W^{(\beta, r)}$ to $W^{(\beta)}$ implies that (with full probability on the event that
$\gamma^{(\beta)}(z)$ exists) in Carath\'eodory topology seen from $z$, the inside of the loop corresponding to $\gamma^{(\beta, r, *)}$ converges to the inside of $\gamma^{(\beta)}(z)$ as $r \to 0$ (note that the loops $\gamma^{(\beta, r, *)}$ will surround $z$ for all small enough $r$).

It is a little more tricky to handle the case where $A_2^{(\beta)} (z)$ holds.
Note that then $\tau^{(\beta)} (z)$ cannot belong to some $(t_- (e), t_+ (e)]$ (because this interval corresponds to a quasi-loop that is locally drawn like a simple slit).
Furthermore, in this case, for any time $t$ before
$\tau^{(\beta)}(z)$, $z$ belongs to the still to be explored unbounded simply connected domain $H_t^{(\beta)}$. Target-independence of the SLE($\kappa, \kappa-6$) processes yields that modulo
reparametrization, one can view the evolution of the Loewner chain  before $\tau^{(\beta)}(z)$ as a radial SLE($\kappa, \kappa-6$) targeting $z$ as described before. 
 In particular, this implies that the conformal radius of $H_t^{(\beta)}$ seen from $z$ decreases continuously to some limit when $ t$ increases to $ \tau^{(\beta)} (z)$. This limit is the conformal radius of the simply connected domain that we denote by $\Omega^{(\beta)} (z)$ and that is the decreasing limit (seen from $z$ in Carath\'eodory topology) of $H_t^{(\beta)}$ as $t$ increases to $\tau^{(\beta)}(z)$ 
(note that we have not excluded the -- somewhat unlikely -- case where $\tau^{(\beta)} (z) = \infty$ when $A_2^{(\beta)} (z)$ holds).

In the next couple of paragraphs, since $z$ is fixed, we will omit to mention the dependence of $\tau^{(\beta)}(z)$, $\Omega^{(\beta)}(z)$ etc on $z$, and write $\tau^{(\beta)}$, $\Omega^{(\beta)}$ instead.

In the same way, we can define almost surely, for each $r$, $\gamma^{(\beta,r)} = \gamma^{(\beta, r)}(z)$ or $\Omega^{(\beta, r)} = \Omega^{(\beta, r)} (z)$ depending on whether
the Loewner chain drawn by $W^{(\beta,r)}$ creates a quasi-loop around $z$ or not (here, we use the target-independence of the symmetric SLE($\kappa, \kappa-6$)). 
One problem to circumvent is that the fact that $W^{(\beta, r)}$ converges uniformly to $W^{(\beta)}$ on any compact interval is not enough to ensure that
$\Omega^{(\beta, r)}$ converges to $\Omega^{(\beta)}$, for instance because just before the disconnection time, a very small fluctuation of the driving function can create a path that  enters and exits this soon-to-be-cut-off domain. An additional problem is that it could be that, even for small $r$, $\tau^{(\beta,r)} (z)$ is much larger than $\tau^{(\beta)}(z)$.

One way around this is to introduce additional stopping times that approximate $\tau^{(\beta)}$ from below. For instance, for each $\epsilon$,
define $\hat \tau_{\epsilon}^{(\beta)} $ to be the first time $t$ at which it is 
possible to disconnect  $z$ from infinity in $H_t$ by removing from $H_t$ a ball of radius $\epsilon$. Then, let $\tilde  \tau_{\epsilon}^{(\beta)} = \min ( 1/ \epsilon , \hat  \tau_{\epsilon}^{(\beta)} )$ (this is just to take care of the possibility that $\tau^{(\beta)} = \infty$).  Then finally, if this time belongs to some quasi-loop interval $(t_- (e), t_+(e))$, define $\tau_{\epsilon}^{(\beta)}$ to be the end-time $t_+ (e)$ of this quasi-loop, and otherwise let
$ \tau_{\epsilon}^{(\beta)} = \tilde \tau_{\epsilon}^{(\beta)} $. Note that $\tau_\eps^{(\beta)}$ is a stopping time with respect to the filtration of $W^{(\beta)}$.

Clearly, for any small $\epsilon$,  $\tau_{\epsilon}^{(\beta)}  < \tau^{(\beta)}$ as soon as $\tau^{(\beta)}$ is finite (recall that $\tau^{(\beta)} $ cannot be equal to some $t_+(e)$).
On the other hand, $\tau_{\epsilon}^{(\beta)}$ cannot increase to anything else than $\tau^{(\beta)} $ as $\epsilon$ decreases to $0$.

Furthermore, the convergence of $W^{(\beta, r)}$ to $W^{(\beta)}$ ensures that for each given $\epsilon$, the domain $H_{\tau_\epsilon^{(\beta)}}^{(\beta, r)}$
converges in Carath\'eodory topology (seen from $z$) to $H_{\tau_\epsilon^{(\beta)}}^{(\beta)}$ when $r \to 0$.
Hence, we conclude that for a well-chosen $r(\epsilon)$ (small enough), as $\epsilon \to 0$ along some well-chosen sequence, almost surely on the even $\tau^{(\beta)}  < \infty$,  the sequence of domains $H_{\tau_\epsilon^{(\beta)}}^{(\beta, r(\epsilon))}$ converges in Carath\'eodory topology to $\Omega^{(\beta)} $.

For each given $\epsilon$, choose $r(\epsilon)$ very small,   we can construct a loop surrounding $z$ as follows, using the Loewner chain associated to $W^{(\beta, r)}$.
If this chain has traced a loop $\gamma^{(\beta,r, \epsilon)}$ that surrounds $z$ along the way before
$\tau^{(\beta)}_\epsilon $, then just keep this loop.
If not, then sample in the
domain $H_{\tau_\epsilon^{(\beta)}}^{(\beta, r(\epsilon))}$ an independent CLE$_\kappa$,  look at the loop that surrounds $z$ in this sample, and call it $\gamma^{(\beta, r, \epsilon)}$. Our previous arguments relating the loops traced by $W^{(\beta, r)}$ to CLE$_\kappa$ (note that $\tau_\epsilon^{(\beta)}$ is a stopping time at which all these Loewner chains  complete a quasi-loop) show that for each fixed $\epsilon$ and $r$, the law of
$\gamma^{(\beta, r, \epsilon)}$ is the same as the law of the loop that surrounds $z$ in an CLE$_\kappa$ in $\HH$.

Note finally that in the case $A_1^{(\beta)}$ where a quasi-loop $\gamma^{(\beta)}$ is discovered by the Loewner chain driven by $W^{(\beta)}$, then (as we have argued before) the same holds true
for the Loewner chain driven by $W^{(\beta, r)}$ for all small $r$, and that furthermore, for all small $\epsilon$ and $r$, $\gamma^{(\beta,r, \epsilon)} = \gamma^{(\beta, r)}$.

Hence, we can conclude that for a well-chosen sequence $\epsilon_n \to 0$ (with a well-chosen $r(\epsilon_n)$),  in the limit when $n \to \infty$, the constructed loop (which has always the law of $\gamma(z)$ in a CLE$_\kappa$ in $\HH$ as we have just argued) is:
\begin {itemize}
 \item Either $\gamma^{(\beta)}$ if $A_1^{(\beta)}$ holds
 \item Or otherwise, obtained by sampling a CLE$_\kappa$ in $\Omega^{(\beta)}$.
\end {itemize}
But now, because we know that this procedure defines a loop that is distributed like $\gamma(z)$ is a CLE$_\kappa$, one can use the branching idea, and iterate this result within $\Omega^{(\beta)}$. We then readily get that the law of the quasi-loop that surrounds $z$ in a CLE$_{\kappa}^{(\beta)}(0)$ is identical to the law of the loop that surrounds $z$ for a CLE$_\kappa$.
\medbreak

In fact, in order to describe the law of CLE$_\kappa^{(\beta)}(0)$, we have to see what happens when one looks at the joint distribution of the loops that surround $n$ given points $z_1, \ldots, z_n$. The argument is
almost identical: The main difference is that one first replaces $\tau^{(\beta)} (z)$ by $\tau^{(\beta)} := \tau^{(\beta)} (z_1, \ldots, z_n) = \min ( \tau^{(\beta)} (z_1), \ldots,
\tau^{(\beta)} (z_n))$ and also changes $\tau_\epsilon^{(\beta)} (z)$ into
$\tau_\epsilon^{(\beta)} := \tau_\epsilon^{(\beta)} (z_1, \ldots, z_n )$ which is associated to the first time at which adding some ball of radius $\epsilon$ to the Loewner chain
 disconnects at least one of these $n$ points from infinity in $H_t^{(\beta)}$. The rest of the procedure is basically unchanged and readily leads to the fact that the joint distribution of the set of loops that surround these points in a CLE$_\kappa^{(\beta)} (x)$ or in a CLE$_\kappa$ are identical, which completes the proof of Proposition \ref {eqvincle}. The only additional ingredient is to note that when one considers a decreasing sequence of open simply connected sets $\Omega_n$, the law of the CLE$_\kappa$ in $\Omega_n$ (as described via its finite-dimensional marginals) converges to the
 law obtained by sampling independent CLE$_\kappa$'s in the different connected components of the interior of $\cap \Omega_n$ (which is a fact that follows for instance directly from the description of  CLE$_\kappa$ as loop-soup cluster boundaries). This ensures that in the limit when $n \to \infty$ (for well-chosen $\eps_n$ and $r_n$), on the event $\tau^{(\beta)} (z_1, \ldots, z_n) < \infty$, 
 sampling a CLE$_\kappa$ in $H_{\tau_{\eps_n}^{(\beta)}}^{(\beta, r_n)}$ converges in law (in the sense of finite-dimensional distributions) to sampling two independent CLE$_\kappa$'s in 
 $H_{\tau^{(\beta)}}^{(\beta)}$ and in the cut-out component $\Omega^{(\beta)}$ at time $\tau^{(\beta)}$.

\medbreak

For $\kappa=4$ and $\mu\in\mathbb{R}$. In the same spirit, we choose the driving function
$$W_t^{<\mu,r>}:= W_t + 4 \mu r N_r (t) =  2B_t+ 4(I_t^{(0)}+\mu r N_r(t)).$$
 Then, the very same arguments as before show that the corresponding constructed loops are those of a CLE.
And on the other hand, the driving function $W^{<\mu,r>}$ converges to $W^{<\mu>}$. This allows to  complete the proof of the following fact:

\begin {proposition}
\label {eqvincle2}
 When $\kappa=4$, the law of CLE$_{4, \mu}(x)$ does depend neither on $x$ nor on $\mu$.
\end {proposition}

\section{The uniform exploration of CLE$_4$}

Let us now modify  the ``symmetric'' Loewner driving function $W^{(0)}$ by introducing some random jumps. Basically, at each time $T_n (r)$ (the end-times of the excursions of $X$ of time-length at least $r^2$), we decide to resample the position of the driving function  according to the uniform density in $[-m, m]$  -- here $m$ should be thought of as very large,  we will then let it go to infinity.

Suppose for a while that $m$ is fixed (we will omit to mention the dependence in $m$ during the next paragraphs in order to avoid heavy notation).
Let us associate to each excursion $e$ of $X$ a random variable $\xi (e)$ with this uniform distribution on $[-m,m]$, in such a way that conditionally on $X$, all these variables $\xi$ are i.i.d. (for notational simplicity, we sometimes also write $\xi=\xi (T)$ when $T$ is the time of $X$ at which the corresponding excursion is finished).
Then, we define the function $t \mapsto \hat W_t^{r}$ as follows: $T_0(r)= 0$ and  for each $n \ge 0$,
\begin {itemize}
\item $\hat W^r (T_n) = \xi (T_{n+1})$.
\item $\hat W^r - W^{(0)}$ is constant on each interval $[T_n (r), T_{n+1} (r))$.
\end {itemize}
The function $\hat W^r$ is piecewise continuous, and it is therefore the driving function of some Loewner chain.
The very same arguments as before show that for each given $r$, it
defines a family of loops distributed like loops in a CLE$_\kappa$ (in the sense that, just as before, when one completes the picture with independent CLE's in the not-yet-filled parts, one obtains a CLE$_\kappa$ sample  (note that the jump distribution -- i.e. the choice of the new point according to the uniform distribution --  is in fact independent of the future behavior of $X$).

It is easy to understand what happens to this construction when $r$ tends to $0$. As before, we are going to look at the almost sure behavior of $\hat W^{r}$ when $r \to 0$, for a given sample of $W^{(0)}$ and $\xi$'s . Let us define the process $\hat W$ by the fact that for each excursion $e$ corresponding to a time-interval $(S,T)$,
$$ \hat W_t = ( W_t^{(0)} - W_S^{(0)} ) + \xi (T) $$
for $t \in [S,T)$ (this defines $\hat W$ for all $t$, except on the zero-Lebesgue measure set of times that are not in the time-span of some excursion, for those times, we can choose $\hat W$ as we wish).
Then, clearly,
the fact that $t \mapsto W^{(0)}_t$ is continuous ensures that for each given excursion interval, $\hat W^{r}$ converges uniformly to $\hat W$ as $r \to 0$ on this time-interval, because for small enough $r$, $\hat W^r = \hat W$ on this excursion.
It follows readily that the Loewner chain generated by $\hat W^r$ does (almost surely) converge (in Carath\'eodory topology) to the one generated by $\hat W$.

Hence, using the same arguments as above (the law of the traced loops is always that of loops in CLE, that are simple disjoint loops, the excursion-intervals correspond to the loops, and these intervals are the same for all $r$), we conclude that during each excursion time-interval, the driving process $\hat W$ does indeed trace a loop, and that the joint law of all these loops are those of loops in a CLE.

Let us now rephrase all the above construction in terms of the Poisson point process of Bessel excursions $(e_u, u \ge 0)$. As we have explained earlier, each Bessel excursion in fact corresponds (via Loewner's equation) to a two-dimensional loop in the upper-half plane, that touches the boundary only at the origin. Let us call $\gamma_u$ the loop
corresponding to $e_u$. To each excursion $e_u$ of the Bessel process, we also associate a random position $x_u \in \R$ sampled according to the uniform measure on $[-m,m]$
(more precisely, conditionally on all Bessel excursions $(e_{u_j})$, the random variables $(x_{u_j})$ are i.i.d. with this distribution).
Then, we define the loop $\hat \gamma_u$ by shifting $\gamma_u$ horizontally by $x_u$ (and so, the loop $\hat \gamma_u$ touches the real axis at $x_u$).

For each excursion $e_u$, we can now define the conformal transformation $\hat f_u$ from the connected component of $\HH \setminus \hat \gamma_u$ that contains $i$ onto $\HH$ such that
$\hat f_u(i)=i$ and $\hat f_u'(i) \in \R_+$. As this will be useful, we now reintroduce the dependence on $m$ in the notation for these maps (and write $\hat f_u = \hat f_u^m$).

For a given $m$, we start with a Poisson point process $(e_u, \ge 0)$ of Bessel excursions defined under the measure $2m \lambda$ and we then associate to each excursion the
uniform random variable $\xi (e_u)$. A cleaner equivalent way to describe the process $((e_u, \xi_u), u \ge 0)$ is to say that it is a Poisson point process with intensity
$\lambda \otimes dx 1_{x \in (-m,m)}$.

Then, clearly, we get a Poisson point process $(\hat f_u^m, u\ge 0)$ of such conformal maps (because for each $u$, $\hat f_u^m$ is a deterministic function of the pair $(e_u, x_u)$ and $((e_u, x_u), u\ge 0)$ is a Poisson point process).

For each $u>0$, one can then define
$$\hat F_u^m = \circ_{v < u} \hat f_v^m$$
(where the composition is done in the order of appearance of the maps $\hat f_v$). Clearly, $\hat F_u^m $ corresponds to the Loewner map (generated by the driving function $\hat W$) at the time (in the Loewner time parametrization) corresponding to the completion of all loops $\hat{\gamma}_v^m$ for $v < u$. In other words, if $\tau (e_u)$ is the time-length of the excursion $e_u$, the Loewner time at which the loop corresponding to that excursion will start being traced is
$\sum_{v < u} \tau (e_v)$.

Hence, the loops
$$\tilde \gamma_u^m := (\hat{F}_u^m)^{-1} ( \hat \gamma_u^m )$$
are distributed like CLE loops.
In particular, the loop that contains $i$ will be the loop $(\hat{F}_{\tau}^m )^{-1} ( \hat{\gamma}^m_\tau )$ where
$$ \tau = \inf \{ u \ge 0 \ : \ \hat{\gamma}^m_u \hbox { surrounds } i \}.$$

Let us rephrase what we have done so far: For each $m$, we have seen that one can define  CLE loops by considering the Loewner chain generated by $\hat W$, using the Poisson point process $(\hat f_u^m, u \ge 0)$, or equivalently, via the Poisson point process $\hat \Gamma^m := (\hat \gamma_u^m, u \ge 0)$ with intensity measure
$$ M^m= \int_{-m}^m dx \mu^x$$
where $\mu^x$ denotes the measure on loops rooted at $x$ (like in \cite {CLE1}, we define this measure as the measure $\mu^0$  on loops in the upper half-plane generated via a Bessel excursion defined under $\lambda$, and shifted horizontally by $x$).

Now, let us describe what happens when $m \to \infty$. Suppose now that we consider the Poisson point process $\hat \Gamma:= (\hat \gamma_u, u \ge 0)$ with intensity
$$M:= \int_{\R} dx \mu^x$$ and the corresponding iterations of maps $\hat F_u$.
Even though this measure seems ``even more infinite'' than $M^m$, this iteration of conformal maps does not explode. This is due to the scaling properties of $\mu^x$ and to the fact that one normalizes always at $i$ (so that loops rooted far away do not contribute much the derivative at $i$) -- one can for instance justify this using Lemma \ref {confinvM} below.

Note also that if we keep only those loops in $\hat \Gamma$ that are rooted at a point in $[-m,m]$, we obtain a process with the same law as $\hat \Gamma^m$.
The key observation is now to see that when $m \to \infty$, each map $\hat F^m_u$ converges uniformly in any compact subdomain of the closed upper half-plane to $\hat F_u$.
This implies the following:

\begin {lemma}
 The loops
$\tilde \Gamma=  (\tilde{\gamma}_u:=\hat F_u^{-1} (\hat \gamma_u), u \le \tau)$
are also distributed like loops in a CLE.
\end {lemma}
\begin{figure}[ht!]
\begin{center}
\begin{subfigure}[b]{0.7\textwidth}
\centering
\includegraphics[width=\textwidth]{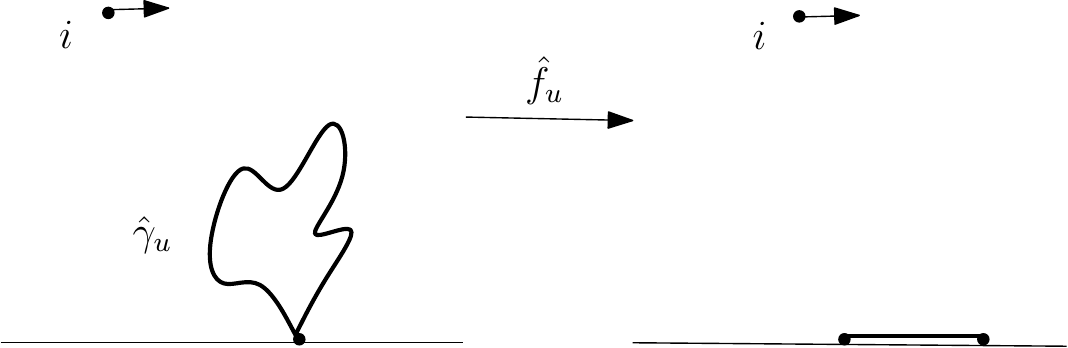}
\caption{$\hat{f}_u$ fixes $i$ and $\hat{f}'_u(i)>0$.}
\end{subfigure}\bigbreak
\begin{subfigure}[b]{0.7\textwidth}
\centering
\includegraphics[width=\textwidth]{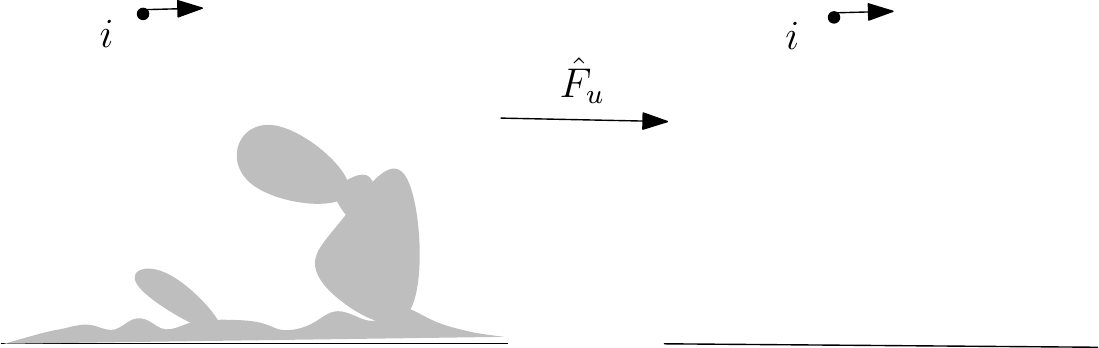}
\caption{$\hat{F}_u$ is the composition of $\hat{f}_v$ for $v<u$.}
\end{subfigure}\bigbreak
\begin{subfigure}[b]{0.7\textwidth}
\centering
\includegraphics[width=\textwidth]{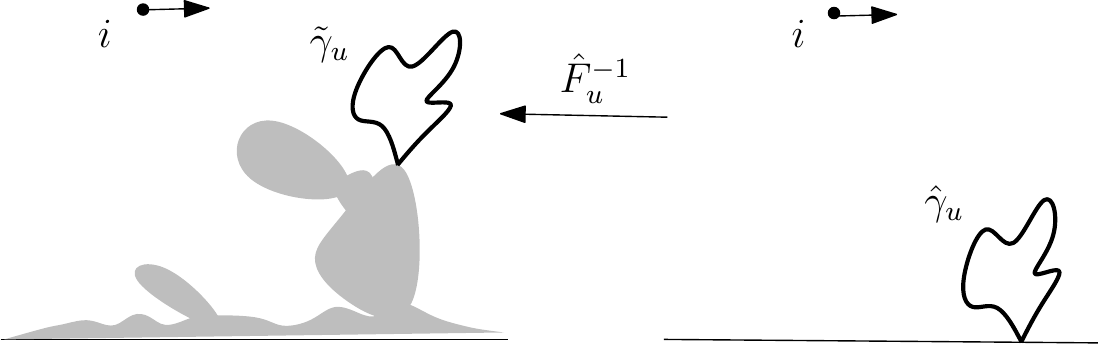}
\caption{$\tilde{\gamma}_u$ is the preimage of $\hat{\gamma}_u$ under $\hat{F}_u$.}
\end{subfigure}
\end{center}
\caption{\label{fig::preimage_bubble} construction of $\tilde{\gamma}_u$.}
\end{figure}

At this stage, everything we have said is still true if we replace the Lebesgue measure on $\R$ by (almost) any other given distribution on $\R$, and any
$\kappa \in (8/3, 4]$. An important reason to choose this particular measure and to focus on the case where $\kappa =4$ is that the following Lemma holds only in this case:
\begin{lemma}
\label {confinvM}
When $\kappa=4$ the measure $M$ is conformal invariant.
\end{lemma}
\begin{proof}
Recall from \cite {CLE1} that when $\kappa =4$ and if $\Phi$ is a conformal transformation of the half-plane onto itself,
$$ \Phi \circ \mu^x  = | \Phi' ( x) | \mu^{\Phi(x)}$$
where the measure $ \Phi \circ \mu^x$ is defined by
$$ \Phi \circ \mu^x (A)  = \mu^x \{ \gamma \ : \ \Phi (\gamma) \in A \}.$$
Hence, it follows immediately that $ \Phi \circ M = M $.
\end{proof}

A direct consequence of this conformal invariance is that
\begin {corollary}
When $\kappa=4$, the law of $\tilde \Gamma= (\tilde \gamma_u, u \le \tau)$ is invariant under any Moebius transformation $\Phi$ of the upper half-plane that preserves $i$.
\end {corollary}

Note that there is no time-change involved. The law of $\Phi ( \tilde \gamma_u) 1_{u \le \tau}$ and $\tilde \gamma_u 1_{u \le \tau}$ are for instance identical.

\begin {proof}
Let $\Phi$ be a Moebius transformation of the upper half-plane that preserves $i,$ and $(\hat{\gamma}_u, u\ge 0)$ be a Poisson point process with intensity $M.$ And define $\tau, \hat{f}_u$ for $u<\tau,$ and $\hat{F}_u, \tilde{\gamma}_u$ for $u\le \tau$ as described above.

Note that $(\bar{\gamma}_u:=\Phi(\hat{\gamma}_u), u\ge 0)$ is a Poisson point process with intensity $M= \Phi\circ M$, and it has therefore the same distribution as
 $(\hat{\gamma}_u, u\ge 0)$.
 For $u<\tau,$ let $\bar{f}_u$ be the conformal map from the connected component of $\HH\setminus\bar{\gamma}_u$ that contains $i$ onto $\HH$ such that $\bar{f}_u(i)=i$ and $\bar{f}'_u(i)\in\R_+.$ It is easy to see that
$$\bar{f}_u=\Phi\circ\hat{f}_u\circ\Phi^{-1}$$
and hence for $u\le\tau,$
$$\bar{F}_u:=\circ_{v<u}\bar{f}_v=\Phi\circ\hat{F}_u\circ\Phi^{-1}.$$
As a result, for $u\le\tau,$
$$\Phi(\tilde{\gamma}_u)=\Phi(\hat{F}_u(\hat{\gamma}_u))=\bar{F}_u(\bar{\gamma}_u).$$
Since $(\bar{\gamma}_u, u\ge 0)$ has the same distribution as $(\hat{\gamma}_u, u\ge 0)$, it follows that $(\Phi(\tilde{\gamma}_u), u\le\tau)$ has the same distribution as $(\tilde{\gamma}_u, u\le\tau).$
\end{proof}

In fact, a stronger result holds. Let us now choose some other point $z$ than $i$ in the upper half-plane.
Let $\sigma$ denote
the first moment if it exists at which the process  $(\tilde \gamma_u, u \le\tau)$ disconnects $i$ from $z$. If the loop $\tilde{\gamma}_{\tau}$ surrounds both $i$ and $z$, we simply set $\sigma=\tau$.  Note that the event that $\sigma < \tau$  can happen when the process discovers a loop surrounding one of the two points and not the other, but at this stage, it is not excluded that it can disconnect two points strictly before discovering the loops that surround them, just like the symmetric SLE($\kappa, \kappa-6$) does, see \cite {CLE1}.

 Define the same process $(\tilde{\gamma}^z_u, u\le \tau^z)$ as above, except that we choose to normalize ``at $z$'' instead of normalizing at $i$.
One way to describe it would be to write  $\tilde \gamma^z_u = \varphi ( \tilde \gamma_u)$, where $\varphi$ is the affine transformation from $\HH$ onto itself such that $\varphi (i) =z$
(but we will use another way to describe it in terms of $\tilde \gamma_u$ in a moment). We then define $\sigma^z$ to be the first moment at which it disconnects $z$ from $i$ or discovers the loop that surrounds both points.

\begin {lemma}
 When $\kappa=4$, and for any $z\in\HH$, the law of $(\tilde \gamma_u^z, u \le \sigma^z)$ is identical to the law of $(\tilde \gamma_u, u \le \sigma)$.
\end {lemma}

We shall use the following classical result about Poisson point process (see for instance \cite {bertoin}, Section 0.5):
\begin{result}
Let $(a_u, u\ge 0)$ be a Poisson point process with some intensity $\nu$ (defined in some metric space $A$). Let  $ \mathfrak{F}_{u-}=\sigma(a_v, v <u)$.  If $(\Phi_u,u\geq 0)$ is a process (with values on functions of $A$ onto $A$) such that for any $u\geq 0,$ $\Phi_u$ is $\mathfrak{F}_{u-}$-measurable, and that $\Phi_u$ preserves $\nu$ then $(\Phi_u(a_u),u \geq 0)$ is still a Poisson point process with intensity $\nu$.
\end{result}

We are now ready to prove the lemma.
\begin{proof}
Consider the process $(\tilde \gamma_u, u \le \tau)$ defined from
the Poisson point process $(\hat{\gamma}_u, u \ge 0)$ with intensity $M$ as above (and keep the same definitions for $\tau$, $\hat{f}_u$ with $u<\tau$, $\hat{F}_u, \tilde{\gamma}_u$ with $u\le\tau$, where the latter are defined by normalizing the maps at $i$).

 We denote  $ \mathfrak{F}_{u-}=\sigma(\hat{\gamma}_v, v <u)$.
For $u<\sigma,$ $\hat{F}_u^{-1}(\HH)$ is a simply connected domain of $\HH$ containing $z$. Let $G_u$ be the conformal map from $\hat{F}_u^{-1}(\HH)$ onto $\HH$ normalized at $z$ by
$G_u(z)=z$ and $G_u' (z) \in \R_+$. Define for each $u<\sigma$
$$\Phi_u=G_u\circ\hat{F}_u^{-1}.$$
We also define
$\Phi_\sigma=\lim_{u\to \sigma-}\Phi_u$, and we say that $\Phi_u$ is the identity map for all $u > \sigma$.
It is clear that, for each positive $u$, the map $\Phi_u$ is a $\mathfrak{F}_{u-}$-measurable Moebius transformation from the upper half-plane onto itself. Hence, the process $(\bar{\gamma}_u:=\Phi_u(\hat{\gamma}_u), u\ge 0)$ is also Poisson point process with intensity $M$.

If we use the point process $(\bar \gamma_u)$ to construct the process $(\tilde \gamma^z_u)$ normalized at $z$, we get a coupling of $(\tilde \gamma_u)$ and
$(\tilde \gamma_u^z)$ in such a way that they coincide up to time $\sigma$: For all $u < \sigma$, $\tilde \gamma_u = \tilde \gamma_u^z$, and in addition,
$$ \bar \gamma_\sigma = \Phi_\sigma ( \hat \gamma_\sigma )$$
(if this $\hat{\gamma}_\sigma$ exists)
so that $\tilde \gamma_\sigma = \tilde \gamma_\sigma^z$.

Hence, with this coupling, we see that $\sigma \le \sigma^z$ almost surely. By symmetry (because there exists a conformal map interchanging these two points), it follows that
$\sigma = \sigma^z$ almost surely.
\end{proof}

\medbreak

This means that it is possible to couple these two processes up to the first moment at which it disconnects $i$ from $z$. By scaling, this shows that for any pair of points $z$ and $z'$, we can couple the two processes $\tilde \gamma^z$ and $\tilde \gamma^{z'}$ up to the first time at which they disconnect $z$ from $z'$.
Hence, it is possible to couple the processes $\tilde \gamma^z$ for {\em all} $z \in \HH$ simultaneously in such a way that for any two points $z$ and $z'$, the previous statement holds.

If we now use such a coupling, we get a Markov process on domains $(D_u, u \ge 0)$: At time $u=0$, the domain is the upper half-plane, and at time $u > 0$, it is the union of all
the (disjoint) open sets corresponding to the evolution to all points $z$ at time $u$. Existence of such a conformally invariant process is a rather striking feature, as it uses no reference point, and the time of the evolution is preserved through the conformal transformation. We can therefore sum up the properties of this process as follows.

\begin {proposition}
\label {propo}
The process $(D_u, u \ge 0)$ provides a way to construct CLE$_4$. Furthermore,
the processes $(D_u, u \ge 0)$ and $(\Phi(D_u), u \ge 0)$ are identically distributed (with no time-change) for all Moebius transformations $\Phi$.
\end {proposition}

\medbreak

Note that this uniform exploration mechanism can also be viewed as the limit (in law) of the asymmetric CLE$_{4,\mu}$ construction of
Proposition \ref {eqvincle2} in the limit when $\mu \to \infty$ (and the boundary points $+\infty$ and $-\infty$ of the upper half-plane are identified, alternatively, one can state this easily in the radial setting).  We leave the details to the interested reader.

\section {Comments and open questions}

\subsection {Some open questions.}
We first mention some natural open questions that are closely related to the present paper:

It is proved in \cite {SS,Du} that an SLE$_4$ can be deterministically drawn as the contour lines (or sometimes called level-lines or cliff-lines) in a Gaussian Free Field with appropriate boundary conditions. See also \cite {MS1} for the fact that the entire CLE$_4$ can be deterministically embedded in a Gaussian Free Field. Note that the symmetric exploration process looks a priori more naturally associated to the Gaussian Free Field than the asymmetric ones, because when one defines a Gaussian Free Field out of a CLE$_4$ with the coupling described in \cite {MS1}, one has to toss an independent coin for each CLE$_4$ loop to decide an orientation, so that the symmetric SLE($4, -2$) (including the coin tosses) is defined via the randomness present in the GFF. However, as shown in \cite {SWW}, the dynamic ``uniform'' construction of SLE$_4$ described in Proposition \ref {propo} actually also yields a rather natural construction of the GFF.

In general, the conformally invariant ways to construct a CLE that we described in the present paper via these branching Loewner chains induce additional information than just the CLE (which loop is discovered where, what is the starting and end-point of the loop when one uses a given exploration etc.).
This leads naturally to the following open questions:

\begin {enumerate}
\item If we are given a CLE$_\kappa$ in a simply connected domain $D$ and a starting point on the boundary of $D$, and a family of Bernoulli random variables $\epsilon (\gamma)$ with parameter $\beta$ (one for each CLE loop $\gamma$), is the asymmetric exploration process with parameter $\beta$ deterministically defined?
\item In particular, is the totally asymmetric exploration process (when $\beta=1$) sample a deterministic function of the CLE sample and of the starting point?
\item Is the uniform exploration process in fact a deterministic function of the CLE$_4$?
\end {enumerate}

A positive answer to this last question would give rise to a conformally invariant distance between loops in a CLE$_4$. This will be investigated further in  the forthcoming paper \cite {SWW}.

\subsection {Discrete explorations}

Our proofs rely a lot on the fact that the symmetric Bessel explorations do indeed construct the loops in a CLE, which was derived in \cite {CLE1} using a discretization of the exploration procedure that was proved to converge to the symmetric Bessel construction. The other CLE constructions that we have studied in the present paper also have natural discrete counterparts that we now briefly describe. However, it turns out to be (seemingly) technically more unpleasant to control the convergence of these asymmetric discrete exploration procedures than the symmetric ones, so that it seemed simpler to derive our results building on the relation between CLE's and the symmetric construction.

We first recall the exploration procedures described in \cite{CLE1} to explore a CLE little by little (here a CLE is just a collection of loops that satisfy the CLE axioms defined in \cite {CLE1}). It will be easier to explain things in the radial setting i.e. in the unit disk instead of the half-plane.

Suppose that $\Gamma=(\gamma_j,j\in J)$ is a CLE in the unit disc $\U$ and that $\epsilon>0$ is given. We denote $\gamma(z)$ as the loop of $\Gamma$ (when it exists) surrounding $z\in\U$.
Throughout this section, $D(1,\epsilon)$ will denote the image of the set $\{z\in\HH: |z|<\epsilon\}$ under the conformal map $\Psi: z\mapsto (i-z)/(i+z)$ from the upper half-plane $\HH$ onto the unit disc such that $\Psi(i)=0,\Psi(0)=1$. Note that for small $\epsilon$, this set is rather close to a small semi-disc centered at $1$.

At the first step, we ``explore'' the small shape $D(1,\epsilon)$ in $\U$, and we discover all the loops in $\Gamma$ that intersect $D(1,\epsilon)$.
If $\gamma(0)$ has already been discovered during this first step, we define $N=1$ and we stop. Otherwise, we
let $U_1$ denote the connected component that contains the origin of the set obtained when removing from
$U'_1=\U\setminus D(1,\epsilon)$ all the loops that do not stay in $U'_1$. From the restriction property in the CLE axioms, the conditional law of $\Gamma$ restricted to $U_1$ (given $U_1$) is that of a CLE in this domain.

We now choose some point $x_1$ on $\partial U_1$, and  the conformal map $\varphi^{\epsilon}_1$ from $U_1$ onto $\U$ such that
  $\varphi^{\epsilon}_1(0)=0$ and $\varphi^{\epsilon}_1 (x_1) = 1$.
Note that we allow here for different possible choices for $x_1$. It can be a deterministic function of $U_1$, but the choice of $x_1$ can also involve additional randomness (we can for instance choose it according to the harmonic measure at the origin etc.), but we impose the constraint that conditionally on $U_1$, the CLE restricted to $U_1$ and the point $x_1$ are conditionally independent (in other words, one is not allowed to use information about the loops in $U_1$ in order to choose $x_1$).

During the second step of the exploration, one discovers the loops of $\Gamma_1:= \varphi^{\epsilon}_1 (\Gamma \cap U_1) $ that intersect $D(1,\epsilon)$.
In other words, we consider the pushforward of $\Gamma$ by $\varphi^\epsilon_1$ (which has the same law as $\Gamma$ itself, due to the CLE axioms) and we repeat step 1. If we discover a loop that surrounds the origin at that step,
then we stop and define $N=2$. Otherwise, we define the connected component $U_2$ that contains the origin of the domain obtained when removing from $\U \setminus D(1,\epsilon)$ the loops of $\Gamma_1$ that do not stay in this domain, and we define the conformal map $\varphi^{\epsilon}_2$ from $U_2$ onto $\U$ with $\varphi^{\epsilon}_2(0)=0$ and $\varphi_2^\epsilon (x_2) = 1$, where $x_2$ is chosen in a conditionally independent way of $\Gamma_1 \cap U_2$, given $U_1$, $U_2$ and $x_1$.

We then explore $\Gamma_2 := \varphi_2^\epsilon (\Gamma_1)$ and so on.
We can iterate this procedure until the step $N$ at which we eventually ``discover" a loop that surrounds the origin. Note that $\gamma(0)$ (the loop in $\Gamma$ that surrounds the origin) is the preimage of this loop (the loop in $\Gamma_N$ that surrounds the origin and intersects $D(1,\epsilon)$) under $\varphi_N^\epsilon \circ \cdots\circ \varphi_1^\epsilon$.

In this definition, the discrete exploration ``strategy'' is encoded by $\epsilon$ (the ``step-size'') and by the rule used to choose the $x_n$'s.
Since the probability to discover the loop at each given step $n$ (conditionally on the fact that it has not been discovered before) is constant and positive, it
follows that $N$ is almost surely finite, that its law is
 geometric (regardless of the choice of $x_n$'s).

 \medbreak

In \cite {CLE1}, it is shown that if a CLE (ie. satisfying the CLE axioms exist), then its loops are of SLE$_\kappa$-type for some $\kappa \in (8/3,4]$ and that it is necessarily the one constructed via the symmetric Bessel construction (and therefore unique). Conversely (using a different argument involving loop-soups) it is shown that these CLE do exist.
The strategy of one part of the proof is to control the behavior of certain natural discrete exploration strategies when $\epsilon$ tends to $0$:

The first one is the ``exploration normalized at the origin''. Here, at each step,  $x_n$ and $\varphi_n^\epsilon$  are chosen according to the rule that
$({\varphi^{\epsilon}_n})'(0)$ is a positive real number. In other words, we choose $\varphi_n^\epsilon$ using the standard normalization at the origin.

Towards the end of the paper \cite {CLE1}, it is shown that the symmetric exploration indeed constructs an axiomatic CLE, by using the following symmetric discrete exploration procedure: Define
$1^+_\epsilon $ and $1^-_\epsilon$
the two intersections of $\partial D(1, \epsilon)$ with the unit circle. At each step, one tosses a (new) fair coin to decide which one of the two points gets mapped conformally onto $1$. In other words, the maps $\varphi_n^\epsilon$ are i.i.d., $x_1$ is independent of $U_1$, and $$ P (x_1 = 1_\epsilon^+)= P( x_1 = 1_\epsilon^-) = 1/2.$$

The definition of the asymmetric discrete explorations is then natural: For a given $\beta$, we toss a $(1+\beta)/2$ vs. $(1-\beta)/2$ coin in order to chose which one of the two points $1_\eps^+$ or $1_\eps^-$ to choose, but in order to compensate the created bias, we post-compose the obtained map $\tilde \varphi_n^\epsilon$ with a deterministic rotation of some angle $\theta ( \eps)$ that vanishes as $\epsilon \to 0$ (that corresponds to the jump in the approximation $I^{(\beta, r)}$ of $I^{(\beta)}$). However, we see that this rotation depends on the chosen base-point (here the origin); this is one reason for which this discrete approximation is a little harder to master than in the case $\beta=0$.

The definition of uniform discrete approximations is also very natural: Just choose $x_n$ at random on the boundary of $U_n$ according to the harmonic measure seen from $0$. Equivalently, choose any $\varphi_n^\epsilon$ and compose it with a uniformly chosen rotation. Again, this rule depends on the target point (the origin) -- but this one is less tricky to control as $\epsilon \to 0$.
We leave it to the interested (and motivated) reader to check that these discrete explorations indeed converge in distribution to the continuous CLE constructions that we have studied in the present paper.

\subsection* {Acknowledgements.}
The research of H.W. is funded by the fondation CFM-JP Aguilar for research. The authors also acknowledge support of the MAC2 project of ANR, as well as the hospitality of
TU Berlin/Einstein Foundation. The authors thank the referee for insightful and very helpful comments on the first versions of this paper.

\smallbreak

H.W. and W.W.:

Laboratoire de Math\'ematiques, Universit\'e Paris-Sud, 91405 Orsay cedex, France

\medbreak

W.W.:

DMA, ENS, 45 rue d'Ulm, 75230 Paris cedex 05, France

\medbreak

hao.wu@math.u-psud.fr

wendelin.werner@math.u-psud.fr

\end{document}